\documentclass[reqno]{amsart}

\begin{document}
\title[Generalized solution]
{ Generalized solution to a system of conservation law which is not strictly hyperbolic}

\author[Sahoo]
{Manas R. Sahoo}

\address{Manas R. Sahoo \newline
TIFR Centre for Applicable Mathematics\\
Sharada Nagar, Chikkabommasandra, GKVK P.O.\\
Bangalore 560065, India}
\email{\newline sahoo@math.tifrbng.res.in}

\thanks{}

\begin{abstract}
In this paper we study a non strictly system of conservation law when viscosity is present and viscosity is zero,
 which is studied in \cite{jm1}. We show the existence and uniqueness of the solution in 
the space of generalized functions of Colombeau for the viscous problem 
and construct a solution to the inviscid system in the sense of association. Also we construct a solution using 
shadow wave approach \cite{ma} and Volpert product which was partly determined as vanishing viscosity limit in \cite{jm1}. 
\end{abstract}

\maketitle
\numberwithin{equation}{section}
\numberwithin{equation}{section}
\newtheorem{theorem}{Theorem}[section]
\newtheorem{remark}[theorem]{Remark}
\newtheorem{lemma}[theorem]{Lemma}
\newtheorem{definition}[theorem]{Defination}
\section{Introduction}

 Conservation laws come in applications are sometime not strictly hyperbolic. Classical theory of Lax\cite{la1} and 
Glimm \cite{17} does not apply in this case. In general such systems do not
admit distributional solutions. As product of distributions arises, one can not search solutions in the space of distributions.
The ideal space where one should search solutions is the Colombeau algebra of generalized functions. 
For detail, see  Colombeau \cite{co2,co3,co4} and Oberguggenberger \cite{ob1}.

Our interest is to study the inviscid partial differential equation 

\begin{equation}
\begin{aligned}
u_t + (\frac{u^2}{2})_x& =0,\,\,\,\
v_t + (uv)_x =0\\
w_t + (\frac{v^2}{2} + uw)_x& =0,\,\,\,\
z_t + (vw + uz)_x =0.
\end{aligned}
\label{e1.1}
\end{equation}
and viscous regularization of \eqref{e1.1} with coefficient of viscosity a generalized constant $\gamma$ ,
\begin{equation}
\begin{aligned}
u_t + (\frac{u^2}{2})_x& =\frac{\gamma}{2} u_{xx} ,\,\,\,\
v_t + (uv)_x =\frac{\gamma}{2} v_{xx}\\
w_t + (\frac{v^2}{2} + uw)_x& =\frac{\gamma}{2 }w_{xx},\,\,\,\
z_t + (vw + uz)_x=\frac{\gamma}{2 }z_{xx},
\end{aligned}
\label{e1.2}
\end{equation}

with initial conditions
\begin{equation} 
(u(x,0), v(x,0), w(x,0), z(x,0))= (u_0, v_0, w_0, z_0)
\label{e1.3}
\end{equation} 
are generalized functions of Colombeau.

The system which we have considered here is $n=4$ case of the following system.
\begin{equation}
(u_j)_t + \sum_{i=1}^{j}(\frac{u_i u_{j-i+1}}{2})_x = \frac{\epsilon}{2} (u_j)_{xx},\,\,j=1,2,...n.
 \label{e1.4}
\end{equation}
 where $\epsilon>0$ is a small parameter. A method to write down
an explicit formula for the solution of \eqref{e1.4}, with initial condition of the form
\begin{equation}
u_j(x,0)=u_{j0}(x), j=1,2,...,n
\label{e1.5}
\end{equation}
is given in \cite{j2}. It is well-known that the corresponding inviscid system
\begin{equation}
(u_j)_t + \sum_{i=1}^{j}(\frac{u_i u_{j-i+1}}{2})_x = 0,\,\,j=1,2,...n.
 \label{e1.6}
\end{equation}
does not have smooth global solution, even if the initial data \eqref{e1.5} is smooth; one has to seek
solution in a weak sense and weak solutions are not unique. Additional conditions are required to pick
the unique physical solution. Vanishing viscosity method is one of the ways to select the physical weak solution
of \eqref{e1.4}. That is, the solution of the inviscid system is constructed
as the limit $\epsilon$ goes to zero of solutions
$u_j^{\epsilon}(x,t)$ of \eqref{e1.4}, with suitable initial conditions.
This was successfully carried out for the cases $n=1$ and $n=2$ for general initial data.
For $n=3$, only partial results are available. In fact it was observed in \cite{j3} that the order of singularity
of vanishing viscosity limit of solutions of \eqref{e1.4} increases as $n$ increases. 

More precisely, when $n=1$, with $u=u_1$, \eqref{e1.4} become the celebrated Burgers equation,
\begin{equation*}
u_t + (\frac{u^2}{2})_x = \frac{\epsilon}{2} u_{xx},
\label{e1.7}
\end{equation*}
which was
explicitly solved for the initial value problem by Hopf \cite{h1} and Cole \cite{co1}.
Hopf \cite{h1} showed that the vanishing viscosity
limit of its solution with a given bounded measurable initial data is a bounded measurable
and locally a BV function which is the weak entropy solution
to the inviscid Burgers equation
\[
u_t + (\frac{u^2}{2})_x = 0\nonumber.
\]
Viscous and inviscid Burgers equation in Colombeau setting has studied in \cite{ob2}.
For $n=2$, with  $u=u_1,v=u_2$ the system \eqref{e1.4} becomes
\begin{equation*}
\begin{aligned}
u_t + (\frac{u^2}{2})_x &=\frac{\epsilon}{2} u_{xx},\,\,\,
v_t + (uv)_x &=\frac{\epsilon}{2} v_{xx},
\end{aligned}
\label{e1.8}
\end{equation*}
in $\{(x,t) : -\infty < x < \infty, t>0 \}$ and is well studied \cite{j1,j3}.
 The case $n=2$ is a one dimensional modelling for the large scale structure formation of universe, see \cite{w1}.

In \cite{j3}, the existence of the solution for the case $n=3$ in Colombeau class is shown for the bounded measurable initial
data and solution to the inviscid case is considered in the sense of association. However this method can not be carried out
for the case $n=4$ as the solution become more and more singular as $n$ increases.
In \cite{jm1}, the vanshing viscosity limit for the Riemann type initial data is studied for the case $n=4$. Calculation of 
vanishing viscosity limit for $n>2$ is an open question for general type initial data. Here our aim is to study viscous system
for the case $n=4$, i.e., \eqref{e1.2} and inviscid case \eqref{e1.1}  with initial datas are generalized functions of Colombeau.

The present paper is organised in the following way. In section 2, we recall some 
definition and results of algebra of generalized functions of Colombeau. In section 3, we discuss existence and uniqueness of the equation \eqref{e1.2}
when initial data belongs to the generalized space of Colombeau and viscosity parameter is a positive generalized constant.
In section 4, we construct a macroscopic solution to the problem \eqref{e1.1} when initial datas are bounded measurable functions.
In section 5, we construct a shadow wave solution and a solution using volpert product for Riemann type initial data when the first
component develops rarefaction. It is remarkable that the solution constructed using Volpert product agrees 
with the vanishing viscosity limit, obtained in the special case when the first component has rarefaction. We conclude in section 6
with some remarks.
  
 \section{The algebra of generalized functions of Colombeau}
In this section we introduce generalized functions of Colombeau in the domain $\Omega_T=\{(x,t): -\infty < x< \infty, 0< t < T\}$, 
for $T>0$, containing the space of bounded distributions, see \cite{co1,co2, co3,ob1}. Here we shall use a version which is sufficient for our purpose. This algebra we denote it by
$\mathcal{G}_{s,g}$, $s$ stands for `simplified'
 and $g$ stands for `global' .

Let
\begin{equation*}
 \mathcal{E}_{s,g}(\Omega_T) =\{(u^{\epsilon})_{0 \leq \epsilon \leq 1}:u^{\epsilon}\in C^{\infty}(\bar{\Omega}_T)\}
\end{equation*}

We define moderate elements of $\mathcal{E}_{s,g}(\Omega_T)$ as

\begin{equation*}
\begin{aligned}
 \mathcal{E}_{M,s,g}(\Omega_T) =\{&(u^{\epsilon})_{0 \leq \epsilon < 1}\in \mathcal{E}_{s,g}(\Omega_T)
: \forall (l,m)\in \mathbb{N}_0^2\,\,\, \exists\,\,\, \ p >0\\  &\textrm{such that}\,\,\, 
{|\frac{\partial ^j}{\partial t ^j}\frac{\partial ^l u}{\partial x^l}|}_{L^{\infty}(\Omega_T)}=O(\epsilon^{-p}) \}
\end{aligned}
\end{equation*}

and null elements by:

\begin{equation*}
\begin{aligned}
 \mathcal{N}_{s,g}(\Omega_T) =\{&(u^{\epsilon})_{0 \leq \epsilon < 1}\in \mathcal{E}_{s,g}(\Omega_T)
: \forall (l,m)\in \mathbb{N}_0^2\,\,\, \forall \,\,\, \ q >0\\  &\textrm{such that}\,\,\, 
{|\frac{\partial ^j}{\partial t ^j}\frac{\partial ^l u}{\partial x^l}|}_{L^{\infty}(\Omega_T)}=O(\epsilon^{q}) \}
\end{aligned}
\end{equation*}

Note that $\mathcal{E}_{M,s,g}(\Omega_T)$ is a differential algebra under component wise addition and 
multiplication. Also $\mathcal{N}_{s,g}(\Omega_T)$ is a differential ideal. So the quotient space 
\begin{equation*}
 \mathcal{G}_{s,g}(\Omega_T)=\frac{\mathcal{E}_{M,s,g}(\Omega_T)}{\mathcal{N}_{s,g}(\Omega_T)}
\end{equation*}
is also a differential algebra, addition and multiplication being defined at coset level.

The space of all bounded distributions on $\bar{\Omega}_T$ can be embedded in this version of algebra and under this embedding the 
product of two bounded smooth functions is preserved.

We say $u \in  \mathcal{G}_{s,g}(\Omega_T)$ admits a distribution $v\in D'(\Omega_T)$ in the sense of association or as macroscopic
aspect if for all test function $\phi \in D(\Omega_T)$ and a representative $(u^{\epsilon})_{\epsilon > 0}$(so for all representatives):
\begin{equation*}
 \displaystyle{\lim_{ \epsilon \rightarrow 0}} \int_{\Omega_T} u^{\epsilon} \phi(x,t)dx dt =\langle {u,\phi}\rangle
\end{equation*}

An element $u \in  \mathcal{G}_{s,g}(\Omega_T)$ is said to be bounded type if there exist a representative 
$(u^{\epsilon})_{\epsilon > 0}$ of $u$ which is bounded.

An element $u \in  \mathcal{G}_{s,g}(\Omega_T)$ is said to be a macroscopic solution or solution in the sense of \emph{association}
 to the the differential equation $L(u)=0$ if $L(u)$ has the macroscopic aspect $0$. 

A generalized function $u \in \mathcal{G}_{s,g}(\Omega_T)$ is said to be a generalized constant if it has a representative which is constant for each 
$\epsilon > 0$.

\section{ Existence and uniqueness}
In this section first we write explicit solutions of \eqref{e1.2} for a representative $\bar{\gamma}$ of $\gamma$.
That is,
\begin{equation}
\begin{aligned}
u_t + (\frac{u^2}{2})_x& =\frac{\bar{\gamma}(\epsilon)}{2} u_{xx},\,\,\,\
v_t + (uv)_x =\frac{\bar{\gamma}(\epsilon)}{2} v_{xx}\\
w_t + (\frac{v^2}{2} + uw)_x& =\frac{\bar{\gamma}(\epsilon)}{2 }w_{xx},\,\,\,\
z_t + (vw + uz)_x =\frac{\bar{\gamma}(\epsilon)}{2 }z_{xx}.
\end{aligned}
\label{e3.1}
\end{equation}
with initial data,
\begin{equation}
 (u(x,0),v(x,0),w(x,0),z(x,0))=(u_0^{\epsilon} (x),v_0^{\epsilon} (x), w_0^{\epsilon}(x),z^{\epsilon}_0(x))
\label{e3.2}
\end{equation}
where $u_0^{\epsilon} (x),v_0^{\epsilon} (x), w_0^{\epsilon}(x),z^{\epsilon}_0(x)$ are representative of $u_0, v_0, w_0, z_0$ 
respectiveily, for each fixed $\gamma(\epsilon)>0$, which can be found in \cite{jm1} with $\gamma(\epsilon)$ replaced by $\epsilon$ .
Before the statement of  the Theorem, we introduce some notations.
Starting from the initial data $u_0^\epsilon,v_0^\epsilon,w_0^\epsilon,z_0^\epsilon$, we define
\begin{equation}
\begin{aligned}
U_0^{\epsilon}(x)&=\int_0^x u_0^{\epsilon}(y) dy,\,\, V_0^\epsilon (x)=\int_0^x v_0^\epsilon(y) dy,\\
W_0^\epsilon(x)&=\int_0^xw_0^\epsilon(y) dy,\,\, Z_0^\epsilon(x)=\int_0^x z_0^\epsilon(y) dy,
\label{e3.3}
\end{aligned}
\end{equation}

and the functions $a,b,c,d$
\begin{equation}
\begin{aligned}
a(x,t)&=\frac{1}{\sqrt{2\pi t \bar{\gamma} (\epsilon)}}{\displaystyle{
  \int_{-\infty}^{+\infty} e^{-\frac{1}{\bar{\gamma} (\epsilon)}[{U}_0^\epsilon(y)
+ \frac{(x-y)^2}{2t}]}dy}},\\
b(x,t)&=-\frac{1}{\bar{\gamma} (\epsilon) \sqrt{2\pi t \bar{\gamma}(\epsilon)}}
 {\displaystyle{\int_{-\infty}^{+\infty}V_0^\epsilon(y)
e^{-\frac{1}{\bar{\gamma}(\epsilon)}[{U}_0^\epsilon(y)
+ \frac{(x-y)^2}{2t}]}dy}},\\
c(x,t)&=\frac{1}{\sqrt{2\pi t \bar{\gamma} (\epsilon)}}{\displaystyle{
 \int_{-\infty}^{+\infty}[\frac{{V_0^\epsilon(y)}^2}{2\bar{\gamma}(\epsilon)^2}-
 \frac{W_0^\epsilon(y)}{\bar{\gamma}(\epsilon)}]  e^{-\frac{1}{\bar{\gamma} (\epsilon)}[{U}_0^\epsilon(y)
+ \frac{(x-y)^2}{2t}]}dy}},\\
 d(x,t)&=\frac{1}{\sqrt{2\pi t \bar{\gamma} (\epsilon)}}{\displaystyle{
 \int_{-\infty}^{+\infty}[-\frac{Z_0^\epsilon(y)}
{\bar{\gamma}(\epsilon)}-\frac{{V_0^\epsilon(y)}^3}{6\bar{\gamma} (\epsilon) ^3}+
 \frac{V_0^\epsilon(y) W_0^\epsilon(y)}{\bar{\gamma} (\epsilon) ^2}]  e^{-\frac{1}{\bar{\gamma} (\epsilon)}[{U}_0^\epsilon(y)
+ \frac{(x-y)^2}{2t}]}dy}}.
\label{e3.4}
\end{aligned}
\end{equation}

\begin{theorem}
For the initial data  $u_0^\epsilon, v_0^\epsilon, w_0^\epsilon$ and $z_0^\epsilon$, there exits a
classical solution of \eqref{e3.1} - \eqref{e3.2}  is given by
\begin{equation}
\begin{aligned}
u^\epsilon &=-\bar{\gamma}(\epsilon)(\log(a))_x,
v^\epsilon =-\bar{\gamma}(\epsilon)(\frac{b}{a})_x,
w^\epsilon =-\bar{\gamma}(\epsilon)(\frac{c}{a}
              -\frac{b^2}{2 a^2})_x,\\
z^\epsilon&= -\bar{\gamma}(\epsilon)(\frac{1}{3}(\frac{b}{a})^3 - \frac{b c}{a^2}+\frac{d}{a})_x,
\label{e3.5}
\end{aligned}
\end{equation}
where $a,b,c$ and $d$ are given by \eqref{e3.4}.
\label{result1}
\end{theorem}

\begin{theorem}
 Let $\gamma$ be a generalized constant with a representative $\bar{\gamma}$ satisfying : there exist
 $N\in \mathbb{N}$ and a $\eta$
such that $\bar{\gamma}({\epsilon})\geq \epsilon^N$ for each $0<\epsilon < \eta$. If the intial data
$(u_0, v_0, w_0, z_0) \in ( G_{s,g}(\mathbb{R}))^4$, then there exists a solution $(u, v, w, z) \in (\mathcal{G}_{s,g}(\Omega_T))^4$
of the equation \eqref{e1.2}-\eqref{e1.3} whose representative explicitly given by \eqref{e3.3}-\eqref{e3.5}.
\end{theorem}
\begin{proof}
 We follow Joseph \cite{j3},  Biagioni and Oberguggenberger \cite{ob2}. We show using representation formula \eqref{e3.5} that
$u^\epsilon ,v^\epsilon,w^\epsilon$ and $z^\epsilon$ satisfy moderate estimates.
 
 So it is enough to show  $\frac{a_x}{a},\frac{b}{a},\frac{c}{a}$ and $\frac{d}{a}$ satisfy moderate estimates.
 
 From the identities \eqref{e3.3} - \eqref{e3.5} it is enough to show the element of the form
 \begin{equation}
 \frac{ \int_{-\infty}^{\infty} k_1 ^{\epsilon}(x,t)
 e^{-\frac{1}{\bar{\gamma}(\epsilon)}[U_0 ^{\epsilon}(y)+\frac{(x-y)^2}{2t}]}dy}
 {\int_{-\infty}^{\infty} e^{-\frac{1}{\bar{\gamma}(\epsilon)}[U_0 ^{\epsilon}(y)+\frac{(x-y)^2}{2t}]}dy}
\in \mathcal{E}_{M,s,g}(\Omega_T)
\end{equation}

where $k_1 ^{\epsilon}$ is a representative of an element of Colombeau class $k_1 \in \mathcal{G}_{s,g}(\Omega_T)$ and 
$U_0 ^{\epsilon}$ is a representative of $U_0 \in \mathcal {\mathcal{G}}_{s,g}(\mathbb{R})$.

Let's denote
\begin{equation}
\begin{aligned}
F_1 (\epsilon, x, t)&=\int_{-\infty}^{\infty} k_1 ^{\epsilon}(x,t)
 e^{-\frac{1}{\bar{\gamma}(\epsilon)}[U_0 ^{\epsilon}(y)+\frac{(x-y)^2}{2t}]}dy\\
 F_2(\epsilon, x, t)&=\int_{-\infty}^{\infty} k_1 ^{\epsilon}(x,t)
 e^{-\frac{1}{\bar{\gamma}(\epsilon)}[U_0 ^{\epsilon}(y)+\frac{(x-y)^2}{2t}]}dy
 \end{aligned}
 \end{equation}
 Then 
 
 \begin{equation}
 \frac{F_1(\epsilon,x, t)}{F_2(\epsilon,x,t)}=\frac{ \int_{-\infty}^{\infty} k_1 ^{\epsilon}(x,t)
 e^{-\frac{1}{\bar{\gamma}(\epsilon)}[U_0 ^{\epsilon}(y)+\frac{(x-y)^2}{2t}]}dy}
 {\int_{-\infty}^{\infty} e^{-\frac{1}{\bar{\gamma}(\epsilon)}[U_0 ^{\epsilon}(y)+\frac{(x-y)^2}{2t}]}dy}
\end{equation}

One can easily show, 
${\partial_x}^{k}(\frac{F_1(\epsilon,x,t)}{F_2(\epsilon,x,t)}$ is the finite linear combinations
of finite products of the elements having the form 
\begin{equation}
 \frac{{\partial_x}^{j_1}F_1(\epsilon,x,t)}{F_2(\epsilon,x,t)}\,\,\, \textnormal{and}\,\,\,
 \frac{{\partial_x}^{j_2}F_2 (\epsilon,x,t)}{F_2 (\epsilon,x,t)},
\end{equation}
where $j_1$ and $j_2$ are nonnegative integers.
Now using change of variable we get
 \begin{equation}
     \begin{aligned}
     {\partial_x}^{j_1}F_1(\epsilon,x,t)&={\partial_x}^{j_1}
\int_{-\infty}^{\infty} k_1 ^{\epsilon}(x,t)
 e^{-\frac{1}{\bar{\gamma}(\epsilon)}[U_0 ^{\epsilon}(y)+\frac{(x-y)^2}{2t}]}dy\\
&= {\partial_x}^{j_1}\int_{-\infty}^{\infty} k_1 ^{\epsilon}(x-z,t)
 e^{-\frac{1}{\bar{\gamma}(\epsilon)}[U_0 ^{\epsilon}(x-z)+\frac{z^2}{2t}]}dz\\
&= \int_{-\infty}^{\infty}P( \bar{\gamma}(\epsilon), {\partial_x}^{1}k_1 ^{\epsilon}(x-z,t),...
{\partial_x}^{j_1}k_1 ^{\epsilon}(x-z,t),{\partial_x}^{1}U_0 ^{\epsilon}(x-z),...\\&{\partial_x}^{j_1}U_0 ^{\epsilon}(x-z))
 e^{-\frac{1}{\bar{\gamma}(\epsilon)}[U_0 ^{\epsilon}(x-z)+\frac{z^2}{2t}]}dz,\\
     \end{aligned}
    \end{equation}
where P is a polynomial of $2J_1+1$ variables. Since the variables satisfy moderate esimates and assumption on $\bar{\gamma}$
imlies $\frac{{\partial_x}^{j_1}F_1(\epsilon,x,t)}{F_2(\epsilon,x,t)})$ satisfy modertate estimate. Similarly one can show
$\frac{{\partial_x}^{j_1}F_2(\epsilon,x,t)}{F_2(\epsilon,x,t)}$  also satisfy moderate estimates. 

So if we take the class $u,v,w,z$ in colombeau space whose representatives are respectiveily 
$u^\epsilon ,v^\epsilon,w^\epsilon,z^\epsilon$, then $u,v,w,z$ satisfy \eqref{e1.2}- \eqref{e1.3}.  
This completes the proof of the theorem.
\end{proof}
Now we show the uniqueness for the Cauchy problem for the equation \eqref{e1.2}. For that we use a modified verson of 
Gronowall inequality from \cite{ob2}.

\begin{lemma}
 Let $u$ be a nonnegative, continuous function on $[0, \infty)$ and assume that
\begin{equation*}
 u(t) \leq a+b\int_{0}^{t}\frac{u(t_1)}{\sqrt{t-t_1}}dt_1 
\end{equation*}
for some constant $a,b \geq 0$ and every $t \geq 0$. Then 
\begin{equation*}
 u(t) \leq a(1+2b\sqrt{t}) \exp(\pi b^2 t)
\end{equation*}
\end{lemma}

\begin{theorem}
 Let $\gamma$ be a generalized constant with a representative $\bar{\gamma}$ satisfying : there exist a $\eta$
such that $\bar{\gamma}(\epsilon)\log(\frac{1}{\epsilon})\geq 1$ for each $0<\epsilon < \eta$. Also assume $u$ is of bounded type.
Then for each $T>0$, the solutions  $u,v,w,z  \in G_{s,g}(R \times [0, T])$ of \eqref{e1.2}-\eqref{e1.3} are unique.
\end{theorem}
\begin{proof}
 Uniqueness for $u$ is already known, see Oberguggenberger \cite{ob2}. 
 Let  $\bar{u}$ is a representative of $u$ which is bounded.

Let $v_1$ and $v_2$ be solutions in the colombeau class with representatives $\bar{v_1}$ and $\bar{v_2}$ 
respectively
Then we will get,
\begin{equation}
\begin{aligned}
 \{(\bar{v_1}-\bar{v_2})_t +(\bar{u}(\bar{v_1}-\bar{v_2}))_x+N\}(\epsilon,x,t)&= 
\bar {\gamma}(\epsilon)(\bar{v_1}-\bar{\rho_2})_{xx}(\epsilon, x,t)\\
(\bar{v_1}-\bar{v_2})&=n(\epsilon,x).
\end{aligned}
\end{equation}
Using Duhamel principle, we have
\begin{equation}
\begin{aligned}
 (\bar{v_1}-\bar{v_2})(\epsilon,x,t) &=\int_{-\infty}^{\infty}G(\epsilon,x,t,x_1,0) n(\epsilon,x_1)dx_1\\
&\int_{0}^{t}\int_{-\infty}^{\infty}G(\epsilon,x,t,x_1,t_1)N(\epsilon,x_1,t_1)dx_1 dt_1\\
 &\int_{0}^{t}\int_{-\infty}^{\infty}\frac{\partial G}{\partial x_1}(\epsilon,x,t,x_1,t_1)
 (\bar{u}(\bar{v_1}-\bar{v_2}))(\epsilon,x_1,t_1)dx_1 dt_1
\end{aligned}
\end{equation}

Since 
$\int_{-\infty}^{\infty}G(\epsilon,x,t,x_1,t_1)dx_1=1$ and 
$\int_{-\infty}^{\infty}\mid\frac{\partial G}{\partial x_1}\mid dx_1= \frac{1}{\sqrt{\pi(t-t_1)\bar{\gamma}(\epsilon)}}$

Thus we obtain the estimate,

\begin{equation}
\begin{aligned}
\displaystyle{\sup_{x}}  \mid{(\bar{v_1}-\bar{v_2})(\epsilon,x,t)}\mid  &\leq 
\displaystyle{\sup_{x_1}}{\mid{n(\epsilon,x_1)}}\mid+
\int_{0}^{t}\displaystyle{\sup_{x_1, t_1}}{ \mid{N(\epsilon,x_1,t_1)}\mid}\\
 &+\int_{0}^{t}\frac{1}{\sqrt{\pi(t-t_1)\bar{\gamma}(\epsilon)}}
\displaystyle{\sup_{x_1}} 
\mid {(\bar{v_1}-\bar{v_2})(\epsilon,x,t)}\mid\\
 &.\displaystyle{\sup_{x,t}} \mid{\bar{u}(\epsilon,x,t)}\mid
 \end{aligned}
\end{equation}

So by lemma $(3.3)$, 

\begin{equation}
\displaystyle{\sup_{(x,t)\in \mathbb{R}\times [0,T]}}  \mid(\bar{v_1}-\bar{v_2})(\epsilon,x,t)\mid  \leq 
a(1+2b\sqrt{T})e^{\pi b^2 T}
\end{equation}
where,
\begin{equation*}
a=\displaystyle{\sup_{x_1}}{\mid{n(\epsilon,x_1)}}\mid+
T\displaystyle{\sup_{(x_1, t_1)\in \mathbb{R} \times [0,T] }} \mid{N(\epsilon,x_1,t_1)}\mid
 \end{equation*} 
 and 

\begin{equation*}
b=\frac{1}{2 \sqrt{\pi \gamma (\epsilon)}}\displaystyle{\sup_{(x_1, t_1)\in \mathbb{R} \times [0,T] }
}\mid\bar{u}(\epsilon,x_1,t_1)\mid
\end{equation*} 
From the assumption of $\gamma (\epsilon)$ we get the condition:
\begin{equation*}
 \displaystyle{\sup_{(x,t)\in \Omega_T}}  \mid{(\bar{v_1}-\bar{v_2})(\epsilon,x,t)}\mid =O(\epsilon^m)
\end{equation*}
for all non negative integers $m$. Since $(\bar{v_1}-\bar{v_2})$ satisfy the above estimate and moderate estimate,
so it is a null element by theorem $(1.2.3)$ of \cite{ob3}

%

 The solution for the component $v$ is unique. Let it has a representative $\bar{v}$. 
 Let $w_1$ and $w_2$ be two solutions for the component 
having representatives $\bar{w_1}$ and $\bar{w_2}$ respectiveily.

 Then we have,
\begin{equation}
 \begin{aligned}
\frac{\partial}{\partial t}\bar{w_1}+ 
\frac{\partial}{\partial x}(\frac{\bar{v}^2}{2}+\bar{u}\bar{w_1})+n_1(\epsilon,x,t)&=
\frac{\gamma(\epsilon)}{2}\frac{\partial^2}{\partial x^2}\bar{w_1}\\
\frac{\partial}{\partial t}\bar{w_2}+ 
\frac{\partial}{\partial x}(\frac{\bar{v}^2}{2}+\bar{u}\bar{w_2})+n_1(\epsilon,x,t)&=
\frac{\gamma(\epsilon)}{2}\frac{\partial^2}{\partial x^2}\bar{w_2},
\end{aligned}
\label{e3.16}
\end{equation}
and 
\begin{equation*}
 (\bar{w_1}-\bar{w_2})=n_3 (\epsilon,x,t)
\end{equation*}

%

where $n_1(\epsilon,x,t)$, $n_2(\epsilon,x,t)$ and $n_3(\epsilon,x,t)$ are null elements in Colombeau space.

Now subtracting the second equation from first in the equation \eqref{e3.16}, we get

\begin{equation}
 \begin{aligned}
\frac{\partial}{\partial t}(\bar{w_1}-\bar{w_2})+ 
+\frac{\partial}{\partial x}\big(\bar{u}(\bar{w_1}-\bar{w_2})\big)+n(\epsilon,x,t)=
\frac{\gamma(\epsilon)}{2}\frac{\partial^2}{\partial x^2}(\bar{w_1}-\bar{w_2}),
\end{aligned}
\end{equation}
where $n(\epsilon,x,t)= n_1(\epsilon,x,t)-n_2(\epsilon,x,t)$ is a null element in Colombeau algebra.
So the analysis similar to above, and condition on $\bar{\gamma}(\epsilon)$ gives the uniqueness for the component $w$.

Uniqueness for the component $z$ can be similarly handled as for the component $w$. 
This completes the proof of the theorem.
\end{proof}

\section{Macroscopic solution of the system}

In this section we show the existence of macroscopic solution of the equation \eqref{e1.1} when the initial data 
$\big(u(x,0), v(x,0), w(x,0), z(x,0)\big)=\big(u_0 (x),v_0 (x),w_0 (x),z_0 (x)\big)$ are bounded measurable functions.
For the case $n=3$, is already considered by Joseph \cite{j3}. That method can not be applied for $z$-component. Here we take a
slower growth order on $u$ to get the required estimates. Now consider the system
\begin{equation}
\begin{aligned}
u_t + (\frac{u^2}{2})_x&=\frac{\beta(\epsilon)}{2} u_{xx},\,\,\,\
v_t + (uv)_x =\frac{\beta(\epsilon)}{2} v_{xx},\\
w_t + (\frac{v^2}{2} + uw)_x &=\frac{\beta(\epsilon)}{2 }w_{xx},\,\,\,\
z_t + (u z + v w)_x = \epsilon z_{xx}.
\end{aligned}
\label{e5.1}
\end{equation}
with initial data
\begin{equation}
 \big(u^{\epsilon}(x,0), v^{\epsilon}(x,0), w^{\epsilon}(x,0), z^{\epsilon}(x,0)\big)=
\big(u_0^{\beta(\epsilon)} (x),v_0^{\beta(\epsilon)}(x),w_0^{\beta(\epsilon)}(x),z_0^{\epsilon} (x)\big),
\end{equation}

where $A^{\epsilon}(x)= A*\eta^\epsilon(x)$, where $\eta^{\epsilon}$ is the usual Friedrich mollifier.
With the notation above, we have following theorem.

\begin{theorem}
Let $u_0,v_0,w_0$ and $z_0$ are bounded measurable functions.Then there exists $\beta(\epsilon)$ such that
  $( u^\epsilon, v^\epsilon, w^\epsilon, z^\epsilon)$ of \eqref{e5.1} satisfy moderate estimates and 
$(u,v,w,z)$ $\in$ $\mathcal{G}_{s,g}(\Omega_T)^4$ correspnding to the representative $( u^\epsilon, v^\epsilon, w^\epsilon, z^\epsilon)$
is a macroscopic solution to 
the system \eqref{e1.1} with initial data \eqref{e1.3}.
\end{theorem}

To prove the above theorem we need the following lemma whose prove can be found [\cite{laso},Chap. I, Theorem 2.5].

\begin{lemma}
Let $u$ satisfy
\begin{equation}
 L(u)=u_t - \displaystyle{\sum _{ij}}a_{ij}(x,t)u_{x_i x_j}+a_{i}(x,t)u_{x_i}+a(x,t)u=f(x,t),
\end{equation}
 where $u$ is continuous at all point $(x,t)\in R^n \times [0, T]$, has continuous derivative $u_t, u_{x_{i}}$ and $u_{x_{i}x_{j}}$ 
satisfies the equation 
for $0< t \leq T$, is bounded, the moduli of the coefficients 
$a_{ij}$,$a_{i}$ do not exceed $c$ and $a(x,t)\geq -a_0$, where $c$ and $a_0$ are nonnegative constants, then the follwing estimate 
holds.

\begin{equation}
 \displaystyle{\sup_{x\in \mathbb{R}^n,0\leq t \leq T}} \mid u(x,t)\mid \leq 
(\displaystyle{\sup_{x\in \mathbb{R}^n}} \mid u(x,0)\mid+ 
T\displaystyle{\sup_{x\in \mathbb{R}^n,0\leq t \leq T}} \mid f(x,t)\mid)\exp(a_0 T)
\label{e5.2}
\end{equation}
 \end{lemma}

{\bf {Proof of the theorem:}}\\
\begin{proof}
 
Applying the lemma to the component $z$ and observing that 
\begin{equation*}
(u^\epsilon)_x = \frac{O(1)}{\sqrt{\beta(\epsilon)}},\,\,\,\,
 (\partial_x)^j v=O(\beta(\epsilon)^{k(j)}),\,\,\,\, 
(\partial_x)^j w=O(\beta(\epsilon)^{l(j)})
\end{equation*} 
for some non negative integers $k(j)$ and $l(j)$.
\begin{equation}
\begin{aligned}
 \displaystyle{\sup_{x\in \mathbb{R},0\leq t \leq T}} \mid z(x,t)\mid &\leq 
(\displaystyle{\sup_{x\in \mathbb{R}}} \mid z(x,0)\mid+ 
T\displaystyle{\sup_{x\in \mathbb{R},0\leq t \leq T}} \mid (vw)_{x}\mid)\exp(\frac{O(1)}{\sqrt{\beta(\epsilon)}})\\
&\leq(\displaystyle{\sup_{x\in \mathbb{R}}} \mid z(x,0)\mid+ 
 O((\beta(\epsilon))^m))\exp(\frac{O(1)}{\sqrt{\beta(\epsilon)}}),
\label{e5.3}
\end{aligned}
\end{equation}
for some negative $m$.
Now chose $\beta(\epsilon)= (\frac{O(1)}{log\frac{1}{\sqrt{\epsilon}}})^2$, then 
 $\exp(\frac{a_0 T}{\beta (\epsilon)})= \frac{1}{\sqrt{\epsilon}}$.

So it is clear that $z$ satisfies
\begin{equation}
 \displaystyle{\sup_{x\in \mathbb{R},0\leq t \leq T}} \mid z(x,t)\mid =O(\frac{1}{\sqrt{\epsilon}})
\label{e4.6}
\end{equation}

Differentiating the fourth equation of \eqref{e5.1}, with  with respect to $x$, we get
\begin{equation*}
(z_x)_t + 2u_x z_x +u (z_x)_x+ (v w)_x+u_{xx}z =\epsilon z_{xxx}.
\end{equation*}

Again applying the lemma to the above equation, we get 
\begin{equation*}
 \displaystyle{\sup_{x\in \mathbb{R},0\leq t \leq T}} \mid z_x (x,t)\mid =O(\frac{1}{\epsilon})
\end{equation*}

 Proceeding inductively and using the equation, we get following estimates on all combination of operators $\partial_t$ and $\partial_x$.
\begin{equation*}
 \displaystyle{\sup_{x\in \mathbb{R},0\leq t \leq T}} \mid \partial_t^l\partial_x^m z (x,t)\mid =O(\frac{1}{\epsilon^r}),
\end{equation*}
for some nonnegative integer $r$ depending only on $l$ and $m$.
 So $(z^\epsilon)$  satisfies moderate estimates. Now we show the colombeau class $u,v,w,z$ corresponding to the representative
$( u^\epsilon, v^\epsilon, w^\epsilon, z^\epsilon)$ is a macroscopic solution to the fourth equation of 
\eqref{e1.1}.

multipying with test function $\phi \in (-\infty, \infty)\times(0, \infty)$ in the equation for component $z$ and integrating
over $(-\infty, \infty)\times(0, T)$, we have

\begin{equation}
 \begin{aligned}
 \int_{0}^{\infty}\int_{-\infty}^{\infty} (z_t + (vw + uz)_x )\phi(x,t)dx dt
&=\epsilon\int_{0}^{\infty}\int_{-\infty}^{\infty}z_{xx}\phi(x,t)dx dt\\
\label{e5.4}
\end{aligned}
\end{equation}
Using integration by parts twice in right hand side of the equation \eqref{e5.4}, we have,

\begin{equation}
 \begin{aligned}
 \int_{0}^{\infty}\int_{-\infty}^{\infty} (z_t + (vw + uz)_x )\phi(x,t)dx dt
 =\epsilon\int_{0}^{\infty}\int_{-\infty}^{\infty}z\phi(x,t)_{xx}dx dt.
\label{e5.5}
\end{aligned}
\end{equation}
 By equation \eqref{e4.6}, right hand side of equation \eqref{e5.4} implies 
\begin{equation*}
  \displaystyle{\lim_{\epsilon \rightarrow 0}}\int_{0}^{\infty}\int_{-\infty}^{\infty}( z_t + (vw + uz)_x) \phi(x,t)dx dt
 =0
\end{equation*}

 By the estimate \eqref{e5.4},the above limit tends to zero as $\epsilon$ tends to zero. 

Since  $\epsilon$ tends to zero imply $\beta(\epsilon)$ tends to zero, we get from \cite{j3} that 
$(u,v,w)$ is a solution to the first three equations of \eqref{e1.1} in the sense of association.

Hence $(u,v,w,z)$ $\in$ $\mathcal{G}_{s,g}(\Omega_T)^4$ correspnding to the representative 
$( u^\epsilon, v^\epsilon, w^\epsilon, z^\epsilon)$
is a macroscopic solution to 
the system \eqref{e1.1}. 
\end{proof}

\section{Explicit solution for Riemann type data}
To understand the macroscopic aspect of the solution obtained in section $3$ is an open question for general initial data. 
However for Riemann type data and when $u$ develops shock or contact discontinuity is completely solved in \cite{jm1}.
But for the case when $u$ develops rarefaction it is partly solved. We describe the results obtained there for this case:

For $u_l <u_r$, $v_l=v_r=\bar{v}$, $w_l=w_r=\bar{w}$, $z_l=z_r=\bar{z}$,

\begin{equation}
\displaystyle{\lim_{\epsilon \rightarrow 0}}( w^{\epsilon},z^{\epsilon}) =
\begin{cases}
 (\bar{w},\bar{z}),\,\,\,if,\,\,\, x\leq u_l t\\
(\dfrac{\bar{v} ^2}{2}t\delta_{x=u_lt},\bar{v}\bar{w}t\delta_{x=u_lt}-\dfrac{\bar{v}^3}{6}t^2\delta'_{x=u_lt})
,\,\,\,if,\,\,\,x=u_lt\\
(0,0),\,\,\,if,\,\,\,u_l t<x<u_r t\\
(-\dfrac{\bar{v} ^2}{2}t\delta_{x=u_rt},-\bar{v}\bar{w}t\delta_{x=u_rt}+\dfrac{\bar{v}^3}{6}t^2\delta'_{x=u_rt})
,\,\,\,if,\,\,\,x=u_rt\\
    (\bar{w},\bar{z}),\,\,\,if,\,\,\, x\geq u_r t
\end{cases}
\label{e4.7}
\end{equation}

It is conjectured there that the distributions
\begin{equation}
\begin{aligned}
&(w(x,t),z(x,t))\\&=
 \begin{cases}
 (w_l,z_l), \,\,\,\,\,\, if \,\,\, x<u_l t\\
 (\dfrac{v_l ^2}{2}t\delta_{x=u_l t},v_lw_l t\delta_{x=u_l t}-\dfrac{v_l ^3}{6}t^2\delta'_{x=u_l t}), \,\,\,\,\,\, if \,\,\, x=u_l t\\
 (0,0),  \,\,\,\,\,\, if \,\,\, u_l t<x<u_r t\\
(-\dfrac{v_r ^2}{2}\delta_{x=u_r t},-v_rw_r t\delta_{x=u_r t}+\dfrac{v_r ^3}{6}t^2\delta_{x=u_r t}), \,\,\,\,\,\, if \,\,\, x=u_r t\\
 (w_r,z_r), \,\,\,\,\,\, if \,\,\, x>u_r t.
 \end{cases}
\label{e4.9}
\end{aligned}
\end{equation}
is the macroscopic aspect when $u$ develops rarefaction.

In this section we construct shadow wave solution\cite{ma} and solution using Volpert product \cite{v1} for Riemann type data when
$u$ develops rarefaction($u_l < u_r$). 
   
First we recall some definition from \cite{ma}. We keep our discussions in a general level.

\begin{definition}
 Let $u_{\epsilon}$ and $u_0$ are given by

\begin{equation}
u^\epsilon (x,t)=\begin{cases}
                    
 u_1,\,\,\,\, if,\,\,\,\, x < (c(t)-\epsilon t)\\

u_{1 \epsilon} \,\,\,\, if ,\,\,\,\,(c(t)-\epsilon t) < x < c(t)\\ 

u_{2 \epsilon} \,\,\,\, if ,\,\,\,\, (c(t)-\epsilon t) < x < (c(t)+\epsilon t)\\ 
 u_1,\,\,\,\, if,\,\,\,\, x > (c(t)+\epsilon t),        
 \end{cases}
\end{equation}

\begin{equation}
u_0 (x)= \begin{cases}
     u_1,\,\,\,\, if,\,\,\,\, x < 0\\ 
u_2,\,\,\,\, if,\,\,\,\, x >0 ,  
        \end{cases}
\end{equation}
where $u_1, u_2, u_{1 \epsilon}$ and $ u_{2 \epsilon}$ are constants and are in ${\mathbb {R}}^n$,
  $(x,t) \in R\times (0,\infty)$ and $f:{\mathbb{R}}^n\rightarrow \mathbb{R}$ is smooth. The line $x=c(t)$ has its
initial point at origin.
Let the distributional limit of $u^{\epsilon}(x,t)$ exists and is $u$.
If $(u^{\epsilon})_t +f(u^\epsilon)_x$  tends to 0, 
in the sense of distribution. 
Then we say $u$ is a Shadow wave solution to the conservation law 
\begin{equation*} 
 u_t +f(u)_x=0
\end{equation*}
with initial data
\begin{equation*} 
 u(x,0)= u_0 (x).
\end{equation*}
\end{definition}
Also there is an entropy concept for this, see \cite{ma}, page[500].
\begin{definition}
Let $\eta(u)$ be a convex entropy with the enropy flux $q(u)$. Then $u^{\epsilon}$ is said to be \emph{entropy admissible}
if 
\begin{equation}
 \displaystyle{\liminf_{\epsilon \rightarrow 0}} \int_\mathbb{R} \int_0 ^T \eta(u^\epsilon)\partial_t \phi dx dt +
\int_\mathbb{R} \eta(u^\epsilon (x,0))dx \geq 0
\end{equation}
for all non-negative test functions $\phi \in C_0 ^\infty (\mathbb{R}\in (\infty, T)$).
\end{definition}

The above definition is equivalent to :
\begin{equation}
 \begin{aligned}
 & \displaystyle{\limsup_{\epsilon \rightarrow 0}} -c(\eta(u_2)-\eta(u_1))+\epsilon (\eta(u_{1\epsilon})+ \eta(u_{1\epsilon}))
+ q(u_2)-q(u_1) \leq 0\\
& \displaystyle{\lim_{\epsilon \rightarrow 0}} -c\epsilon (\eta(u_{1\epsilon})+ \eta(u_{1\epsilon}))+
\epsilon (q(u_{1\epsilon})+ q(u_{1\epsilon}))=0\\
 \end{aligned}
\label{def2}
\end{equation}

Now we construct shadow wave solution in the following theorem.
\begin{theorem}
 If $u_l < u_r$, then a shadow wave solution to the equation \eqref{e1.2} with initial data
\begin{equation*}
( u(x,0),v(x,0),w(x,0),z(x,0))=\begin{cases}
 (u_l, v_l,w_l, z_l),\,\,\,\, if,\,\,\,\, x < 0\\
(u_r, v_r,w_r, z_r),\,\,\,\, if,\,\,\,\, x > 0
                             \end{cases}
\end{equation*}
 is given by
\begin{equation}
(w,z)=\begin{cases}
                    
 ( w_l, z_l),\,\,\,\, if,\,\,\,\, x < u_l t\\
 
\frac{v_l ^2}{2}t \delta _{x=u_l t}, v_l w_l t  \delta _{x=u_l t},if,\,\,\,\, x = u_l t \\

 (0,0) \,\,\,\, if ,\,\,\,\,u_l t < x < u_r t\\ 
-\frac{v_r ^2}{2}t \delta _{x=u_r t}, - v_r w_r t  \delta _{x=u_r t},\,\,\,\,if,\,\,\,\, x = u_r t \\
(w_r, z_r),\,\,\,\, if,\,\,\,\, x > u_r t
\end{cases}
\end{equation}
The Solution is entropy admissible.
\end{theorem}

\begin{proof}

By vanishing viscosity limit, see \cite{jm1} the limit $(u,v)$ for the rarfaction case of $u$ is 
given by

\begin{equation}
(u,v)=\begin{cases}
                    
 (u_l, v_l),\,\,\,\, if,\,\,\,\, x < u_l t\\

 (\frac{x}{t},0) \,\,\,\, if ,\,\,\,\,u_l t < x < u_r t\\ 
(u_r, v_r),\,\,\,\, if,\,\,\,\, x > u_r t
\end{cases}
\end{equation}
 So we guess the following ansatz for $(u,v,w, z)$ for the possible shadow wave approximation.
\begin{equation}
(u_{\epsilon},v_{\epsilon},w_{\epsilon},z_{\epsilon})(x,t)=\begin{cases}
                    
 (u_l,v_l,w_l, z_l),,\,\,\,\, if,\,\,\,\, x < (u_l-\epsilon) t\\
(u_l,\frac{v_1}{\sqrt{\epsilon}},\frac{w_1}{\epsilon},\frac{z_1}{\epsilon}) 
\,\,\,\, if ,\,\,\,\, (u_l-\epsilon) t < x < u_l t\\ 

(\frac{x}{t},0,0,0) \,\,\,\, if ,\,\,\,\, u_l t < x < u_r t\\ 

(u_r,\frac{v_2}{\sqrt{\epsilon}},\frac{w_2}{\epsilon},\frac{z_2}{\epsilon}) 
\,\,\,\, if ,\,\,\,\,   u_r t < x <(u_r+\epsilon) t \\ 
(u_r,v_r,w_r, z_r),,\,\,\,\, if,\,\,\,\, x > (u_r+\epsilon) t
\end{cases}
\end{equation}
Applying formula $3.2$ from \cite{ma}, with $a_\epsilon=\epsilon$, $b_\epsilon=0$, $c=u_l$ 
near the discontinuity line $x=u_lt$ and $a_\epsilon=0$, $b_\epsilon=\epsilon$, $c=u_r$ near
the discontinuity line $x=u_rt$, we get
  
\begin{equation}
\begin{aligned}
 w_t & \approx (u_l w_l +w_1)\delta_{x=u_l t}-u_l w_1 t \delta'_{x=u_l t}\\&+ (-u_r w_r+w_2)\delta_{x=u_r t}
-u_r w_2 t \delta'_{x=u_r t}
\end{aligned}
\end{equation}
\begin{equation}
\begin{aligned}
 \partial_x (\frac{{v^\epsilon}^2}{2}+u^\epsilon w^\epsilon)&\approx
(-\frac{v_l ^2}{2}-u_l w_l)\delta_{x=u_l t}+(-\frac{v_1 ^2}{2}+u_l w_1)t\delta'_{x=u_l t}\\
&+(-\frac{v_r ^2}{2}-u_r w_r)\delta_{x=u_r t}+(-\frac{v_2 ^2}{2}+u_r w_2)t\delta'_{x=u_r t}
\end{aligned}
\end{equation}

The relation  $w_t+\partial_x (\frac{{v^\epsilon}^2}{2}+u^\epsilon w^\epsilon)\approx 0$ implies
 
\begin{equation*}
 v_1=v_2=0, \,\,\,\, w_1= \frac{v_l ^2}{2},\,\,\,\, w_2= -\frac{v_r ^2}{2}.
\end{equation*}
 
 Now we calculate the distributional limit of $w^{\epsilon}$. Let $\phi$ be a real valued test function supported in 
$(-\infty, \infty) \times (0, \infty)$.
\begin{equation}
\begin{aligned}
 &\int_{0}^\infty \int_{-\infty}^\infty w^{\epsilon}(x,t)\phi(x,t)dx dt\\&=
 \int_{0}^\infty \int_{-\infty}^{(u_l-\epsilon)t} w^{\epsilon}(x,t)\phi(x,t)dx dt+
\int_{0}^\infty \int_{(u_l-\epsilon)t}^{u_l t} w^{\epsilon}(x,t)\phi(x,t)dx dt\\&+
\int_{0}^\infty \int_{(u_r)t}^{(u_r+\epsilon)t} w^{\epsilon}(x,t)\phi(x,t)dx dt+
\int_{0}^\infty \int_{(u_r+\epsilon)t}^{\infty} w^{\epsilon}(x,t)\phi(x,t)dx dt\\
&= \int_{0}^\infty \int_{-\infty}^{(u_l-\epsilon)t} w_l\phi(x,t)dx dt+
\int_{0}^\infty \int_{(u_l-\epsilon)t}^{u_l t} \frac{v_l ^2}{2 \epsilon}\phi(x,t)dx dt\\&-
\int_{0}^\infty \int_{(u_r)t}^{(u_r+\epsilon)t}\frac{v_l ^2}{2 \epsilon} \phi(x,t)dx dt+
\int_{0}^\infty \int_{(u_r+\epsilon)t}^{\infty} w_r \phi(x,t)dx dt
\end{aligned}
\end{equation}

As $\epsilon$ tends to 0, we have 
\begin{equation}
\begin{aligned}
 &\displaystyle {\lim_{\epsilon\rightarrow 0}}\int_{0}^\infty \int_{-\infty}^\infty w^{\epsilon}(x,t)\phi(x,t)dx dt\\&=
 \int_{0}^\infty \int_{-\infty}^{u_l t} w_l\phi(x,t)dx dt+
\int_{0}^\infty  \frac{v_l ^2}{2 }\phi(u_l t,t) dt\\&-
\int_{0}^\infty  \frac{v_r ^2}{2}\phi(u_r t,t) dt+
\int_{0}^\infty \int_{u_r t}^{\infty} w_r \phi(x,t)dx dt\\
&= \int_{0}^\infty \int_{-\infty}^\infty w(x,t)\phi(x,t)dx dt
\end{aligned}
\end{equation}
Proceeding as above we get,
\begin{equation}
\begin{aligned}
 z_t & \approx (u_l z_l +z_1)\delta_{x=u_l t}-u_l z_1 t \delta'_{x=u_l t}\\&+ (-u_r z_r+z_2)\delta_{x=u_r t}
-u_r z_2 t \delta'_{x=u_r t}
\end{aligned}
\end{equation}
\begin{equation}
\begin{aligned}
 \partial_x (v^\epsilon w^\epsilon+u^\epsilon z^\epsilon)&\approx
(-v_l w_l-u_l z_l)\delta_{x=u_l t}+\epsilon 
(-\frac{v_1 w_1}{\epsilon^{\frac{3}{2}}}+\frac{u_l z_1}{\epsilon})t\delta'_{x=u_l t}\\
&+(-v_r w_r-u_r z_r)\delta_{x=u_r t}+\epsilon 
(-\frac{v_2 w_2}{\epsilon^{\frac{3}{2}}}+\frac{u_r z_2}{\epsilon})t\delta'_{x=u_r t}\\
\end{aligned}
\end{equation}
From the calculation for $w$, we had $v_1=v_2=0$, So
\begin{equation*}
\begin{aligned}
 \partial_x (v^\epsilon w^\epsilon+u^\epsilon z^\epsilon)&\approx
(-v_l w_l-u_l z_l)\delta_{x=u_l t}+ 
u_l z_1 t\delta'_{x=u_l t}\\
&+(v_r w_r+u_r z_r)\delta_{x=u_r t}+
u_r z_2 t\delta'_{x=u_r t}
\end{aligned}
\end{equation*}

The relation  $z_t+\partial_x (v^\epsilon w^\epsilon+u^\epsilon z^\epsilon)\approx 0$ implies
 
\begin{equation*}
  z_1= v_l w_l,\,\,\,\, z_2= -v_r w_r.
\end{equation*}

 Following the calculation as in $w^\epsilon$, we get
\begin{equation}
 \displaystyle {\lim_{\epsilon\rightarrow 0}}\int_{0}^\infty \int_{-\infty}^\infty z^{\epsilon}(x,t)\phi(x,t)dx dt=
\int_{0}^\infty \int_{-\infty}^\infty z(x,t)\phi(x,t)dx dt.
\end{equation}
To show the solution is \emph{entropy admissible}, we first determine enropy and entropy flux pair for the system \eqref{e1.2}.

Note that if we take the transformation $(u,v,w,z) \rightarrow (2u, v,4w, 24z)$, 
the system \eqref{e1.2} transforms to prolonged system $n=4$
with $f(u)= u^2$, see \cite{p}. Convex entropy for such a system is given by, see \cite{p}.
\begin{equation*}
\eta(u)=\bar{\eta}(u)+c_1 v + c_2 w + c_3 z,
\end{equation*}
where $\eta(u)$ is convex.

Since the transformation $(u,v,w,z) \rightarrow (2u, v,4w, 24z)$ is linear, so convex entropy and flux of the \eqref{e1.2} is :
\begin{equation*}
\begin{aligned}
\eta(u)&=\bar{\eta}(u)+c_1 v + c_2 w + c_3 z\\
q&= \int u\eta'(u)u du +c_1 uv +c_2(\frac{v^2}{2}+uw)+c_3(vw+uz)
\end{aligned}
\end{equation*}
Since 
$u^\epsilon _t+u^\epsilon u^\epsilon_x$, $v^\epsilon_t +(u^\epsilon v^\epsilon)_x$,
$w^\epsilon _t +(\frac{{v^\epsilon}^2}{2}+u^\epsilon w^\epsilon)_x$ and 
$z^\epsilon_t +(v^\epsilon w^\epsilon+u^\epsilon z^\epsilon)_x$ tends to $0$ in the sense of distribution. 

So the entropy condition \eqref{def2} reduces to the usual entropy condition for the first component $u$. 
This completes the proof of the theorem.
 \end{proof}

Now we construct a solution to the problem \eqref{e1.1} using volpert product for the component $w$ and $z$.

\begin{theorem}
 Under volpert product consideration, the solution for the component w and z of the
equation \eqref{e1.1}, when $u$ develops rarfaction with initial data  
\begin{equation*}
( u(x,0),v(x,0),w(x,0),z(x,0))=\begin{cases}
 (u_l, v_l,w_l, z_l),\,\,\,\, if,\,\,\,\, x < 0\\
(u_r, v_r,w_r, z_r),\,\,\,\, if,\,\,\,\, x > 0
                             \end{cases}
\end{equation*}
is given by,

\begin{equation}
 (w,z)=\begin{cases}
                    
 ( w_l, z_l),\,\,\,\, if,\,\,\,\, x < u_l t\\
 
\frac{v_l ^2}{2}t \delta _{x=u_l t}, v_l w_l t  \delta _{x=u_l t}-
(\frac{v_l ^3}{6}+ct^{\frac{1}{2}})\delta'_{x=u_l t},if,\,\,\,\, x = u_l t \\

 (0,0) \,\,\,\, if ,\,\,\,\,u_l t < x < u_r t\\ 
-\frac{v_r ^2}{2}t \delta _{x=u_r t}, - v_r w_r t  \delta _{x=u_r t}+
(\frac{v_r ^3}{6}+ct^{\frac{1}{2}})\delta'_{x=u_r t},\,\,\,\,if,\,\,\,\, x = u_r t \\
(w_r, z_r),\,\,\,\, if,\,\,\,\, x > u_r t,
\end{cases}
\label{e3.13*}
\end{equation}
for arbitary real number $c$. Here $\delta' = \frac{\partial}{\partial x}$.
\end{theorem}

\begin{proof}
 Lets take the following ansatz for $w$ and $z$:
\begin{equation*}
\begin{aligned}
 w(x,t)&= w_l H(u_l t- x)+ w_r (1-H(u_r t-x))+e_l (t)\delta_{x=u_lt} +e_r (t)\delta_{x=u_rt}\\
z(x,t)&=z_l H(u_l t- x)+ z_r (1-H(u_r t-x))+g_l (t)\delta_{x=u_lt} +g_r (t)\delta_{x=u_rt}\\
&+ h_l (t)\delta'_{x=u_rt} +h_r (t)\delta'_{x=u_rt}
\end{aligned}
\end{equation*}
 Note that,
\begin{equation}
\begin{aligned}
\frac{\partial}{\partial t} (a(t)\delta_{x=c t})&=a'(t)\delta_{x=c t}-c a(t) \delta'_{x=c t}\\
\frac{\partial}{\partial t} (a(t)\delta'_{x=c t})&= a'(t)\delta_{x=c t}-c a(t) \delta''_{x=c t}
\label{e3.13}
\end{aligned}
\end{equation}

we get
\begin{equation*}
\begin{aligned}
 \frac{\partial}{\partial t} w &= u_l w_l \delta_{x=u_lt} - u_r w_r \delta_{x=u_lt} \\
&+ e_l '(t)\delta_{x=u_l t}-u_l e_l (t) \delta'_{x=u_l t}+e_r '(t)\delta_{x=u_r t}-u_r e_r (t) \delta'_{x=u_r t}
\end{aligned}
\end{equation*}

\begin{equation*}
\frac{\partial}{\partial x} (uw) =- u_l w_l \delta_{x=u_lt} + u_r w_r \delta_{x=u_lt} \\
+ u_l e_l (t) \delta'_{x=u_l t}+u_r e_r (t) \delta'_{x=u_r t}
\end{equation*}
\begin{equation*}
 \frac{\partial}{\partial x} (\frac{v^2}{2}) = -\frac{v_l ^2}{2} \delta_{x=u_lt}+\frac{v_r ^2}{2} \delta_{x=u_rt}
\end{equation*}
Putting all these in third equation of \eqref{e1.1}, we get, 
\begin{equation*}
e_l '(t)=\frac{v_l ^2}{2},\,\,\,\, e_r '(t)=-\frac{v_r ^2}{2}
\end{equation*}
Since at time $t=0$, there is no concentration, we take $e_l (0)=0$, $e_r (0)=0$. 
 So we get 
\begin{equation*}
e_l (t)=\frac{v_l ^2}{2}t,\,\,\,\, e_r (t)=-\frac{v_r ^2}{2}t
\end{equation*} 
Now we calculate for the component $z$:

Using \eqref{e3.13},

\begin{equation*}
\begin{aligned}
 \frac{\partial}{\partial t} z &= u_l z_l \delta_{x=u_lt} - u_r z_r \delta_{x=u_lt} \\
&+ g_l '(t)\delta_{x=u_l t}-u_l g_l (t) \delta'_{x=u_l t}+g_r '(t)\delta_{x=u_r t}- u_r g_r  (t) \delta'_{x=u_r t}\\
&+h_l '(t)\delta'_{x=u_l t}-u_l h_l (t) \delta''_{x=u_l t}+h_r '(t)\delta'_{x=u_r t}-u_r h_r (t) \delta''_{x=u_r t},
\end{aligned}
\end{equation*}

\begin{equation*}
\begin{aligned}
\frac{\partial}{\partial x} (uz)& =- u_l z_l \delta_{x=u_lt} + u_r z_r \delta_{x=u_lt} \\
&+ u_l g_l (t) \delta'_{x=u_l t}-\frac{h_l (t)}{2t}\delta'_{x=u_lt}+u_r g_r (t) \delta'_{x=u_r t}\\
&+u_l h_l (t) \delta''_{x=u_l t}-\frac{h_r (t)}{2t}\delta'_{x=u_rt}+u_r h_r (t) \delta''_{x=u_r t}
\end{aligned}
\end{equation*}
Using Volpert product \cite{v1},
\begin{equation*}
\frac{\partial}{\partial x} (vw) =- v_l w_l \delta_{x=u_lt} + v_r w_r \delta_{x=u_rt} 
+ \frac{v_l^3}{4} \delta'_{x=u_lt} -\frac{v_r^3}{4} \delta'_{x=u_rt}
\end{equation*}
Putting all these in fourth equation of \eqref{e1.1}, we get

\begin{equation*}
g_l '(t)=v_l w_l,\,\,\,\, g_r '(t)=-v_r w_r,\,\,\,\,
h_l '(t)-\frac{h_l (t)}{2t}=\frac{v_l ^3}{4}t,\,\,\,\,h_r '(t)-\frac{h_r (t)}{2t}=\frac{v_r ^3}{4}t
\end{equation*}
Since at time $t=0$, there is no concentration, we take $g_l (0)=g_r (0)=h_l (0)=h_r (0)=0$ 
 So we get 
\begin{equation*}
g_l (t)=v_l w_l t,\,\,\,\,g_r (t)=v_r w_r t,\,\,\,\,
h_l (t)=\frac{v_l ^3}{6}t^2+ct^{\frac{1}{2}},\,\,\,\, h_r (t)=-\frac{v_r ^3}{6}t^2+ct^{\frac{1}{2}}
\end{equation*} 
\end{proof}
\section{Conclusion}

 For $u_l< u_r, v_l=v_r, w_l=w_r, z_l=z_r$, \eqref{e3.13*} becomes
\begin{equation*}
\begin{aligned}
 w(x,t)&= \bar{w} H(u_l t- x)+ \bar{w}(1-H(u_r t-x))+\delta_{x=u_lt} +\frac{\bar{v} ^2}{2}t\delta_{x=u_lt}-
\frac{\bar{v}^2}{2}t\delta_{x=u_rt}\\
z(x,t)&=\bar{z} H(u_l t- x)+ \bar{z} (1-H(u_r t-x))+\bar{v}\bar{w} t\delta_{x=u_lt} -\bar{v}\bar{w} t\delta_{x=u_rt}\\
&- (\frac{\bar{v} ^3}{6}t^2+ct^{\frac{1}{2}})\delta'_{x=u_lt} +(\frac{\bar{v} ^3}{6}t^2+ct^{\frac{1}{2}})\delta'_{x=u_rt}
\end{aligned}
\end{equation*}
Under Volpert product consideration the solution is not unique due to arbitary $c$. 
When $c=0$, \eqref{e3.13*} is exactly the vanishing viscosity limit \eqref{e4.7} which is obtained in \cite{jm1}. 
The shadow wave solution for the component $w$ agrees with the vanshing viscosity limit where as it does not agree for the component $z$. 
Macroscopic aspect of the vanishing viscosity approximation is still an open question for the general type 
initial data.


\begin{thebibliography}{0}
\bibitem{co1}
 J.D.Cole, 
\emph{On a quasi-linear parabolic equation occurring in aerodynamics},
 Quart. Appl. Math,{\bf {9}} (1951) 225-236.

\bibitem{co2}J.F. Colombeau,
\emph{New Generalized Functions and Multiplication of Distributions},
 Amsterdam:North Holland (1984).

\bibitem{co3} J.F. Colombeau, 
\emph{New Generalized Functions and Multiplication of Distributions:A 
graduate course, 
 application to theoretical and numerical solutions of partial differential equations},
 (Lyon)(1993).

\bibitem{co4} J.F. Colombeau and A. Heibig, 
 \emph{Generalized solutions to Cauchy problems}, 
  {Monatsh.math.}, {\bf{117}} (1994), 33-49. 

\bibitem{ma} M. Nedeljov,
\emph {Shadow waves: Entropies and Interactions for delta and singular solutions}
{Arch.Rational Mech.Anal.},{\bf{197}}(2010),489-537.

 
\bibitem{17} J. Glimm,
 \emph{Solution in the large for nonlinear hyperbolic system of equations},
 comm. pure Appl Math.{\bf18}(1965), 697-715.  


%

\bibitem{j1}
 K. T. Joseph, 
\emph{A Riemann problem whose viscosity solution
contain $\delta$- measures.}, Asym. Anal., {\bf{7}} (1993), 105-120 .

\bibitem{j2}
 K. T. Joseph and A. S. Vasudeva Murthy, 
\emph{Hopf-Cole transformation to
some systems of partial differential equations}, NoDEA Nonlinear Diff. Eq. Appl., {\bf{8}} (2001), 173-193 .

\bibitem{j3}
 K.T.Joseph, 
\emph{Explicit generalized solutions to a system of 
conservation laws}, 
Proc. Indian Acad. Sci. Math. {\bf {109}} (1999), 401-409.


\bibitem{jm1}
K.T.Joseph and Manas R. Sahoo,
\emph{Vanishing viscosity approach to a system of conservation laws admitting $\delta''$},
 Commun.pure.Appl.Anal., {\bf{12}} (2013), no. 5, 2091-2118.





\bibitem{ob1} M. Oberguggenberger,
\emph{Multiplication of distributions and Applications to PDEs},
Pittman Research Notes in Math,Longman, Harlow {\bf 259} (1992).

\bibitem{ob2} H.A. Biagioni and M. Oberguggenberger,
\emph {Generalized solutions to Burgers equation},
 J. Differential equations, {\bf{97}} (1992), 263-287.

\bibitem{ob3} M. Grosser, M. Kunzinger, M. Oberguggenberger, S. Roland,
\emph{Geometric theory of generalized functions with applications to general relativity}
Mathematics and its applications, 537. Kluwer Academic Publishers, Dordrecht, (2001).

\bibitem{laso} O. A. Ladyzenskaja, V. A. Solonnikov and N. N. Ural'ceva, 
\emph {Linear and quasilinear equations of parabolic type}
in “Translations Math. Monographs,” Vol. 23, Amer. Math.
Sot., Providence, RI, 1968.



\bibitem{v1}
A.I.Volpert,
\emph{The space BV and quasi-linear equations},  Math. USSR 
Sb., {\bf 2} (1967), 225-267.

\bibitem{h1}
E. Hopf, The partial differential equation $u_t +u u_x =\mu u_{xx}$,
Comm. Pure Appl. Math., {\bf 13} (1950) 201-230.

\bibitem{la1}
P.D. Lax, Hyperbolic systems of conservation laws II, Comm.Pure 
Appl. Math., {\bf 10} (1957) 537-566.

\bibitem{w1}
D.H.Weinberg and J.E.Gunn, Large scale structure and the adhesion
approximation, {\it Mon. Not. R. Astr. Soc.}, {\bf 247} (1990), 260-286.

\bibitem{p} E. Yu. Panov, On a representation of the prolonged systems for a scalar conservation law and 
on higher order entropies, Differential equations {\bf{44}} (2008), 1694-1699.
\end{thebibliography}
 \end{document}